\documentclass[10pt]{amsart}

\usepackage{amssymb,latexsym}
\usepackage{amsfonts}
\usepackage{amsthm}
\usepackage{enumerate}

\newcommand{\w}{\omega}
\newcommand{\U}{\mathcal U}
\newcommand{\C}{\mathcal C}
\newcommand{\V}{\mathcal V}
\newcommand{\F}{\mathcal F}
\newcommand{\Fr}{{\mathfrak F}r}
\newcommand{\II}{\mathbb I}

\newcommand{\rest}{\upharpoonright}
\newcommand{\inchi}{\mathrm{nw}_\chi}
\newcommand{\hinchi}{\mathrm{hnw}_\chi}
\newcommand{\nw}{\mathrm{nw}}
\newcommand{\Ra}{\Rightarrow}

\newtheorem{theorem}{Theorem}
\newtheorem{corollary}{Corollary}
\newtheorem{proposition}{Proposition}
\newtheorem{lemma}{Lemma}
\newtheorem{question}{Question}
\theoremstyle{definition}
\newtheorem{remark}{Remark}
\newtheorem{example}{Example}

\title[On function spaces and convergence in topological spaces]{On meager function spaces, network character and meager convergence in topological spaces}
\author[T.~Banakh, V.~Mykhaylyuk, L.~Zdomskyy]{Taras Banakh, Volodymyr Mykhaylyuk, Lyubomyr Zdomskyy}
\address[T.Banakh]{Ivan Franko National University of Lviv, Universytetska 1, Lviv 79000, Ukraine and\newline Uniwersytet Humanistyczno-Przyrodniczy Jana Kochanowskiego, Kielce, Poland.}
\email{tbanakh@yahoo.com}
\urladdr{http://www.franko.lviv.ua/faculty/mechmat/Departments/Topology/bancv.html}
\address[V.Mykhaylyuk]{Department of Mathematics,
Yuriy Fedkovych Chernivtsi National University,
 Kotsjubynskogo str. 2, Chernivtsi 58012, Ukraine.}
\email{vmykhaylyuk@ukr.net}

\address[L.Zdomskyy]{Kurt G\"odel Research Center for Mathematical Logic,
University of Vienna, W\"ahringer Stra\ss e 25, A-1090 Wien,
Austria.}
\email{lzdomsky@logic.univie.ac.at}
\urladdr{http://www.logic.univie.ac.at/\~{}lzdomsky/}
\thanks{The third  author   acknowledges the support of FWF grant
P19898-N18.}


\begin{document}
\begin{abstract} For a non-isolated point $x$ of a topological space $X$ let $\inchi(x)$ be the smallest cardinality of a family $\mathcal N$ of infinite subsets of $X$ such that each neighborhood $O(x)\subset X$ of $x$ contains a set $N\in\mathcal N$. We prove that
\begin{itemize}
\item each infinite compact Hausdorff space $X$ contains a non-isolated point $x$ with  $\inchi(x)=\aleph_0$;
\item for each point $x\in X$ with $\inchi(x)=\aleph_0$ there is an injective sequence $(x_n)_{n\in\w}$ in $X$ that $\F$-converges to $x$ for some meager filter $\F$ on $\w$;
\item if a functionally Hausdorff space $X$ contains an $\F$-convergent injective sequence for some meager filter $\F$, then for every path-connected space $Y$ that contains two non-empty open sets with disjoint closures, the function space $C_p(X,Y)$ is meager.
\end{itemize}
Also we investigate properties of filters $\F$ admitting an injective $\F$-convergent sequence in $\beta\w$.
\end{abstract}
\maketitle

This paper was motivated by a question of the second author who asked  if the function space $C_p(\w^*,2)$ is meager. Here $\w^*=\beta\w\setminus\w$ is the remainder of the Stone-\v Cech compactification of the discrete space of finite ordinals $\w$ and $2=\{0,1\}$ is the doubleton endowed with the discrete topology.
According to Theorem 4.1 of \cite{Myh} this question is tightly connected with the so-called meager convergence of sequences in $\w^*$.

A filter $\F$ on $\w$ is {\em meager} if it is meager (i.e., of the first Baire category) in the power-set $\mathcal P(\w)=2^\w$ endowed with the usual compact metrizable topology. By the Talagrand characterization \cite{Tal}, a free filter $\F$ on $\w$ is meager if and only if $\xi(\F)=\Fr$ for some finite-to-one function $\xi:\w\to \w$.
A function $\xi:\w\to\w$ is {\em finite-to-one} if for each point $y\in \w$ the preimage $\xi^{-1}(y)$ is finite and non-empty. A filter $\F$ on $\w$ is defined to be {\em $\xi$-meager} for a surjective function $\xi:\w\to\w$ if $\xi(\F)=\Fr$.

We shall say that for a filter $\F$ on $\w$, a sequence $(x_n)_{n\in\w}$ of points of a topological space $X$ {\em $\F$-converges} to a point $x_\infty\in X$ if for each neighborhood $O(x_\infty)\subseteq X$ of $x_\infty$ the set $\{n\in\w:x_n\in O(x_\infty)\}$ belongs to the filter $\F$. Observe that the usual convergence of sequences coincides with the $\Fr$-convergence for the Fr\'echet filter $\Fr=\{A\subseteq\w:\w\setminus A$ is finite$\}$ that consists of all cofinite subsets of $\w$. The filter convergence of sequences has been actively studied both in Analysis \cite{FA1}, \cite{FA2} and Topology \cite{FT1}. A sequence $(x_n)_{n\in\w}$ will be called {\em meager-convergent} if it is $\F$-convergent for some meager filter $\F$ on $\w$. A sequence $(x_n)_{n\in\w}$ is called {\em injective} if $x_n\ne x_m$ for all $n\ne m$.

We shall prove that for a zero-dimensional Hausdorff space $X$ the function space $C_p(X,2)$ is meager if $X$ contains an injective meager-convergent sequence.   We recall that a topological space $X$ is {\em functionally Hausdorff} if for any distinct
points $x,y\in X$ there is a continuous function $\lambda:X\to\II$ such that $\lambda(x)\ne\lambda(y)$. Here $\II=[0,1]$ is the unit interval. For topological spaces $X,Y$ by $C_p(X,Y)$ we denote the space of continuous functions endowed with the topology of pointwise convergence.

\begin{theorem}\label{t2} Let $X$ be a functionally Hausdorff space and $Y$ be a topological space that contains two open non-empty subsets with disjoint closures. Assume that $X$ is zero-dimensional or $Y$ is path-connected. If $X$ contains an injective meager-convergent sequence, then the function space $C_p(X,Y)$ is meager.
\end{theorem}

\begin{proof} Let $(x_n)_{n\in\w}$ be a sequence in $X$ that $\F$-converges to $x_\infty\in X$ for some meager filter $\F$ in $\w$. Then there is a finite-to-one surjection $\xi:\w\to\w$ such that $\xi(\F)=\Fr$. By our assumption, $Y$ contains two non-empty open subsets $W_0,W_1$ with disjoint closures.

For every $n\in\w$ consider the subset
$$\C_{n}=\big\{f\in C_p(X,Y):\forall i\in\{0,1\}\;\big(f(x_\infty)\notin \overline{W}_i\Rightarrow
\forall m\ge n\;\exists k\in\xi^{-1}(m)\;\;(f(x_k)\notin \overline{W}_i)\big)\big\}.$$

The meager property of $C_p(X,Y)$ will follow as soon as we check that $C_p(X,Y)=\bigcup_{n\in\w}\C_n$ and each set $\C_n$ is nowhere dense in $C_p(X,Y)$.

To show that $C_p(X,Y)=\bigcup_{n\in\w}\C_n$, fix any continuous function $f\in C_p(X,Y)$. Since $Y=(Y\setminus\overline{W}_0)\cup (Y\setminus\overline{W}_1)$, there is $i\in\{0,1\}$ such that $f(x_\infty)\notin\overline{W}_i$. Since $(x_n)$ is $\F$-convergent to $x_\infty$ and $f^{-1}(Y\setminus \overline{W}_i)$ is an open neighborhood of $x_\infty$, the set $F=\{n\in\w:f(x_n)\notin\overline{W}_i\}$ belongs to the filter $\F$ and thus the image $\xi(F)$, being cofinite in $\w$, contains the set $\{m\in\w:m\ge n\}$ for some $n\in\w$. Then $f\in \C_n$ by the definition of the set $\C_n$.

Next, we show that each set $\C_n$ is nowhere dense in $C_p(X,Y)$. Fix any non-empty open set $\U\subseteq C_p(X,Y)$. Without loss of generality, $\U$ is a basic open set of the following form:
$$\U=\{f\in C_p(X,Y):\forall z\in Z\;\;f(z)\in U_z\}$$for some finite set $Z\subseteq X$ and non-empty open sets $U_z\subseteq Y$, $z\in Z$. We can additionally assume that $x_\infty\in Z$. We need to find a non-empty open set $\V\subseteq C_p(X,Y)$ such that $\V\subseteq \U\setminus \C_n$. If $\U\cap \C_n$ is empty, then put $\V=\U$. So we assume that $\U\cap \C_n$ contains some function $f_0$. For this function we can find $i\in\{0,1\}$ such that $f_0(x_\infty)\notin \overline{W}_i$. Since $f_0(x_\infty)\in U_{x_\infty}$, we lose no generality assuming that $U_{x_\infty}\subseteq Y\setminus\overline{W}_i$.

Since the sequence $(x_n)_{n\in\w}$ is injective, we can find $m\ge n$ such that the set $X_m=\{x_k:k\in\xi^{-1}(m)\}$ does not intersect the finite set $Z$. Choose any function $g:Z\cup X_m\to Y$ such that $g(z)=f_0(z)$ for all $z\in Z$ and $g(x)\in W_{1-i}$ for all $x\in X_m$.

We claim that the function $g$ has a continuous extension $\bar g:X\to Y$. By our assumption, $X$ is zero-dimensional or $Y$ path-connected. In the first case we
 can find a retraction $r:X\to Z\cup X_m$ and put $\bar g=g\circ r$. If $Y$ is path-connected, then take any
injective function $\phi:g(Z\cup X_m)\to \II$ and extend the  function $\phi\circ g:Z\cup X_m\to\II$ to
a  continuous map $\lambda:X\to\II$ using the functional Hausdorff property of $X$. Since $Y$ is path-connected,
the map $\phi^{-1}:(\phi\circ g)(Z\cup X_m)\to Y$ extends to a continuous map $\psi:\II\to Y$. Then the continuous map
$\bar g=\psi\circ\lambda:X\to Y$ is a required continuous extension of $g$.

In both cases the set $$\V=\{f\in C_p(X,Y):\forall z\in Z\;f(z)\in U_z,\mbox{ and }\forall x\in X_m\;f(x)\in W_{1-i}\}$$ is an open neighborhood of $\bar g$ that lies in $\U\setminus \C_n$, witnessing that the set $\C_n$ is nowhere dense in $C_p(X,Y)$.
\end{proof}

In light of Theorem~\ref{t2} it is important to detect topological spaces that contain injective meager-convergent sequences. This will be done for spaces containing  points with countable network character.


A family $\mathcal N$ of subsets of a topological space $X$ is called a {\em$\pi$-network} at a point $x\in X$ if each neighborhood $O(x)\subset X$ of $x$ contains some set $N\in\mathcal N$. If each set $N\in\mathcal N$ is infinite, then $\mathcal N$ will be called an {\em i-network} at $x$. An i-network at $x$ exists if and only if each neighborhood of $x$ in $X$ is infinite. In this case let $\inchi(x;X)$ denote the smallest cardinality $|\mathcal N|$ of an i-network $\mathcal N$ at $x$. If some neighborhood of $x$ in $X$ is finite, then let $\inchi(x;X)=1$.
If the space $X$ is clear from the context, then we write $\inchi(x)$ instead of $\inchi(x;X)$ and call this cardinal the {\em network character} of $x$ in $X$. If $X$ is a $T_1$-space, then $\inchi(x)\ge\aleph_0$ if and only if the point $x$ is not isolated in $X$. The cardinal
$\hinchi(x)=\sup\{\inchi(x;A):x\in A\subset X\}$ is called the {\em hereditary network character} at $x$. Points $x\in X$ with $\hinchi(x)\le\aleph_0$ are called {\em Pytkeev points}, see \cite{MT}.

\begin{theorem}\label{t1} If some point $x$ of a topological space $X$ has $\inchi(x)=\aleph_0$, then for each finite-to-one function $\xi:\w\to\w$ with $\lim_{n\to\infty}|\xi^{-1}(n)|=\infty$ there is an injective sequence $(x_n)_{n\in\w}$ in $X$ that $\F$-converges to $x$ for some $\xi$-meager filter $\F$.
\end{theorem}

\begin{proof} Let $(N_i)_{i\in\w}$ be a countable i-network at $x$. Since each set $N_i$ is infinite, we can choose an injective sequence $(x_k)_{k\in\w}$ in $X$ such that for every $n\in\w$ and $0\le i<|\xi^{-1}(n)|$ the set $N_i$ meets the set $\{x_k:k\in\xi^{-1}(n)\}$.

It is clear that the sequence $(x_n)_{n\in\w}$ $\F$-converges to $x$ for the filter $$\F=\big\{\{n\in\w:x_n\in O(x)\}:\mbox{$O(x)$ is a neighborhood of $x$ in $X$}\}\big\}.$$ It remains to check that the filter $\F$ is $\xi$-meager. Given any neighborhood $O(x)\subset X$ of $x$ we need to find $n\in\w$ such that for every $m\ge n$ there is $k\in\xi^{-1}(m)$ with $x_k\in O(x)$. Since $(N_i)_{i\in\w}$ is a network at $x$, there is $i\in\w$ such that $N_i\subset O(x)$. Taking into account that $\lim_{n\to\infty}|\xi^{-1}(n)|=\infty$, find $n\in\w$ such that $|\xi^{-1}(m)|> i$ for all $m\ge n$. Now the choice of the sequence $(x_k)$ guarantees that for every $m\ge n$ there is $k\in\xi^{-1}(m)$ with $x_k\in N_i\subset O(x)$.
\end{proof}

In light of Theorem~\ref{t1} it is important to detect points $x$ with countable network character $\inchi(x)$. Let us recall that the {\em character} $\chi(x)$ (resp. the {\em $\pi$-character} $\pi\chi(x)$) of a point $x$ in a topological space $X$ is equal to the smallest cardinality of a neighborhood base (resp. a $\pi$-base) at $x$. A {\em $\pi$-base} at $x$ is any $\pi$-network at $x$ consisting of non-empty open subsets of $X$.
These definitions imply the following simple:

\begin{proposition}\label{p1n} For any non-isolated point $x$ of a $T_1$-space $X$,\begin{enumerate}
\item $\inchi(x)\le\chi(x)$;
\item $\inchi(x)\le\pi\chi(x)$ provided that $x$ has a neighborhood containing no isolated point of $X$;
\item $\inchi(x)=\aleph_0$ if $x$ is the limit of an injective $\Fr$-convergent sequence in $X$.
\end{enumerate}
\end{proposition}

The following simple example shows that the usual convergence of the injective sequence in Proposition~\ref{p1n}(3) cannot be replaced by the meager convergence. It also shows that Theorem~\ref{t1} cannot be reversed.

\begin{example} Let $\F$ be the meager filter on $\w$ consisting of the sets $F\subset\w$ such that $$\lim_{n\to\infty}\frac{|F\cap[2^n,2^{n+1})|}{2^n}=1.$$ On the space $X=\w\cup\{\infty\}$ consider the topology in which all points $n\in\w$ and isolated while the sets $F\cup\{\infty\}$, $F\in\F$, are neighborhoods of $\infty$. It is clear that the sequence $x_n=n$, $n\in\w$, $\F$-converges to $\infty$ in $X$. On the other hand, a simple diagonal argument shows that $\inchi(\infty;X)>\aleph_0$.
\end{example}

\begin{theorem}\label{t3} Each infinite compact Hausdorff space $X$ contains a point $x\in X$ with $\inchi(x)=\aleph_0$.
\end{theorem}

\begin{proof} Theorem trivially holds if $X$ contains a non-trivial convergent sequence. So we assume that $X$ contains no non-trivial convergent sequence.
Then $X$ contains a closed subset $C\subset X$ that admits a continuous map $g:C\to \II$ onto the unit interval $\II=[0,1]$, see \cite[p.172]{Efim}. Replacing $C$ by a smaller subset, we can assume that the map $g:C\to \II$ is irreducible, which means that $g(C')\ne \II$ for any proper closed subset $C'\subset C$. Fix any countable base $\mathcal B$ of the topology of $\II$. The irreducibility of the map $g:C\to\II$ implies that the space $C$ has no isolated points. Also the irreducibility of $g$ implies that the countable family $\mathcal N=\{g^{-1}(U):U\in\mathcal B\}$ of open infinite subsets of $C$ is an $i$-network at each point $x\in C$. Consequently, $\nw_\chi(x)=\aleph_0$ for each point $x\in C$.
\end{proof}

Theorems~\ref{t2}---\ref{t3} imply:

\begin{corollary} For each infinite zero-dimensional compact Hausdorff space $X$  and each topological space $Y$ containing two non-empty open sets with disjoint closures the function space $C_p(X,Y)$ is meager. In particular, the function space $C_p(\w^*,2)$ is meager.
\end{corollary}

Also Theorems~\ref{t1} and \ref{t3} imply

\begin{corollary}\label{c3} Let $\xi:\w\to\w$ be a finite-to-one function with $\lim_{n\to\infty}|\xi^{-1}(n)|=\infty$. Each infinite compact Hausdorff space $X$ contains an injective $\F$-convergent sequence for some $\xi$-meager filter $\F$ on $\w$.
\end{corollary}

In fact, the condition $\lim_{n\to\infty}|\xi^{-1}(n)|=\infty$ in Corollary~\ref{c3}  can not be weakened.

Let us recall that an infinite subset $A$ is called a {\em pseudointersection} of a family of sets $\F$ if $A\subseteq^* F$ for all $F\in\F$ where $A\subseteq^* F$ means that $A\setminus F$ is finite. If a sequence $(x_n)_{n\in\w}$ in a topological space $\F$-converges to a point $x_\infty$ for some filter $\F$ with infinite pseudointesection $A\subseteq \w$ then the subsequence $(x_k)_{k\in A}$ converges to $x_\infty$ in the standard sense.

\begin{lemma}\label{l1} Let  $I$ be a countable set and
 $C=\bigcup_{i\in I}C_i$, where the sets $C_i$ are nonempty and mutually disjoint, and
$\sup_{i\in I}|C_i|<\w$. If ${\mathcal H}$ is a filter on $C$ all of whose elements intersect all but finitely many
$C_i$'s, then ${\mathcal H}$ has an infinite pseudointersection.
\end{lemma}

\begin{proof} The proposition will be proved by induction on $n=\sup_{i\in I}|C_i|$. If $n=1$ there is nothing to prove.
Suppose that it is true for all $k<n$ and let $I$, $\{C_i:i\in I\}$, ${\mathcal H}$ be as above with
$\max\{|C_i|:i\in I\}=n$. If for every $H\in{\mathcal H}$ the set
$\{i\in I: |C_i\cap H|<n\}$ is finite, then $C$ itself is a pseudointersection of ${\mathcal H}$.
So suppose that $J=\{i\in I:|C_i\cap H_0|<n\}$ is infinite for some $H_0\in{\mathcal H}$.
In this case we may use our  inductive hypothesis
for $J$, $\{C_i\cap H_0:i\in J\}$, $\mathcal G={\mathcal H} \rest (\bigcup_{i\in J}C_i\cap H_0)$,
and $n-1$. Thus $\mathcal G$ has an infinite pseudointersection, and hence so does ${\mathcal H}$.
\end{proof}

\begin{proposition}\label{p1} If $\F$ is a $\xi$-meager filter on $\w$ for some surjective function $\xi:\w\to\w$ with  $\underline{\lim}_{n\to\infty}|\xi^{-1}(n)|<\infty$, then any sequence $(x_n)_{n\in\w}$ in a topological space $X$ that  $\F$-converges to a point $x_\infty\in X$ contains a subsequence $(x_{n_k})_{k\in\w}$ that converges to $x_\infty$.
\end{proposition}

\begin{proof} Choose infinite set $I\subseteq \omega$ such that $\sup_{i\in I}|\xi^{-1}(i)|<\w$. Let $C_i=\xi^{-1}(i)$ for every $i\in I$, $C=\bigcup_{i\in I}C_i$ and ${\mathcal H}=\{F\cap C: F\in{\mathcal F}\}$. According to Lemma~\ref{l1} there exists an infinite set $D\subseteq C$ such that $D\subseteq^* H$ for every $H\in\mathcal H$. Then the subsequence $(x_i)_{i\in D}$ converges to $x_{\infty}$.
\end{proof}

Now let us compare two facts:
\begin{enumerate}
\item the compact Hausdorff space $\beta\w$ contains no injective $\Fr$-convergent sequences;
\item each infinite compact Hausdorff space $X$ contains an injective $\F$-convergent sequence for some meager filter $\F$.
\end{enumerate}
These two facts suggest a problem of finding the borderline between filters $\F$ that admit an injective $\F$-convergent sequence in $\beta\w$ and filters that admit no such sequences. We hope that this borderline passes near analytic filters. Let us recall the definitions of some properties of filters.

A filter $\F$ is {\em analytic} (resp. an {\em $F_\sigma$-filter}, {\em $F_{\sigma\delta}$-filter}) if $\F$ is an analytic (resp. $F_\sigma$-subset, $F_{\sigma\delta}$-subset) of the power-set $\mathcal P(\w)=2^\w$ endowed with the natural compact metrizable topology.

A filter $\F$ is {\em measurable} (resp. {\em null}\/) if is it measurable (resp. has measure zero) with respect to the Haar measure on the Cantor cube $2^\w$ considered as the countable product of 2-element groups. It is well-known that a filter is measurable if and only if it is null. The relations between meager and null filters are not trivial and were investigated in \cite{Tal} and \cite{BGJS}. Since each analytic filter is meager and null we get the following chain of properties of filters:
$$\mbox{$F_\sigma$}\;\Ra\;\mbox{analytic}\;\Ra\;\mbox{meager \&\ null}.$$

We are going to show that some meager and null filter $\F$ admits an injective $\F$-convergent sequence in $\beta\w$
while no $F_\sigma$-filer $\F$ admits such a sequence. The latter fact holds more generally for analytic $P^+$-filters.

A filter $\F$ on $\w$ is called a {\em $P$-filter} (resp. a {\em $P^+$-filter}) if each countable subfamily $\C\subset\F$ has a pseudointersection $A$ that belongs to $\F$ (resp. to $\F^+$). Here
$$\F^+=\{A\subset\w:\forall F\in\F\;A\cap F\ne\emptyset\}$$
coincides with the union of all filters that contain $\F$. It is clear that each $P$-filter is a $P^+$-filter. In particular, the Fr\'echet filter $\F$ is both a $P$-filter and $P^+$-filter.

For a filter $\F$ on $\w$ by $\chi(\F)$ we denote its {\em character}. It is equal to the smallest cardinality $|\mathcal B|$ of the base $\mathcal B\subset\F$ that generates $\F$ in the sense that $\F=\{F\subset\w:\exists B\in\mathcal B\;\;B\subset F\}$. It is well-known that the character of each free ultrafilter on $\w$ is uncountable. The uncountable cardinal $\mathfrak u=\min\{\chi(\U):\U\in\beta\w\setminus\w\}$ is called the {\em  ultrafilter number}, see \cite{vD}, \cite{Vau}.
The {\em dominating number} $\mathfrak d$ is the smallest cardinality $|D|$ of a cofinal subset $D$ in the partially ordered set $(\w^\w,\le)$, see \cite{vD}, \cite{Vau}. By Ketonen's Theorem \cite{Ket}, {\em each filter $\F$ on $\w$ with character $\chi(\F)<\mathfrak d$ is a $P^+$-filter}.

Now we can establish some properties of filters $\F$ admitting injective $\F$-convergent sequences in $\beta\w$.

\begin{theorem}\label{verner} Assume that a filter $\F$ admits an injective $\F$-convergent sequence $(x_n)_{n\in\w}$ in $\beta\w$.
\begin{enumerate}
\item If $\F$ is a $P^+$-filter, then for some set $A\in\F^+$ the filter $\F|A=\{F\cap A:F\in\F\}$ on $A$ is an ultrafilter.
\item $\chi(\F)\ge\min\{\mathfrak d,\mathfrak u\}$;
\item $\F$ is not an analytic $P^+$-filter;
\item $\F$ is not an $F_\sigma$-filter.
\end{enumerate}
\end{theorem}

\begin{proof} 1. Assume that $\F$ is a $P^+$-filter. Let $x_\infty$ be the $\F$-limit of the $\F$-convergent sequence $(x_n)_{n\in\w}$ in $\beta\w$.  Since the sequence $(x_n)$ is injective, there is $m\in\w$ such that for every $n\ge m$ \ $x_n\ne x_\infty$ and hence we can fix a neighborhood $U_n$ of $x_\infty$ whose closure does not contain the point $x_n$.  Since the sequence $(x_k)$ $\F$-converges to $x_\infty$, for every $n\ge m$ the set $F_n=\{k\in\w:x_k\in U_n\}$ belongs to the filter $\F$. Since $\F$ is a $P^+$-filter, the sequence $(F_n)_{n\ge m}$ has a pseudointersection $A\in\F^+$. It follows from the choice of the neighborhoods $U_n$ that the set $\{x_n\}_{n\in A}$ is discrete in $\beta\w$ and the sequence $(x_n)_{n\in A}$ is $\F|A$-convergent to $x_\infty$. By Rudin's Theorem~\cite{Rudin}, the map $f:A\to \beta\w$, $f:n\mapsto x_n$, has injective Stone-\v Cech extension $\beta f:\beta A\to\beta\w$, which implies that the filter $\F|A$ is an ultrafilter.
\smallskip

2. If $\chi(\F)<\min\{\mathfrak d,\mathfrak u\}$, then $\chi(\F)<\mathfrak d$ and by the Ketonen's Theorem \cite{Ket} $\F$ is a $P^+$-filter.
By the preceding statement, $\F|A$ is an ultrafilter for some set $A\in\F^+$. Consequently,
$$\mathfrak u\le\chi(\F|A)\le\chi(\F)<\mathfrak u$$and this is a desired contradiction.
\smallskip

3. If $\F$ is an analytic $P^+$-filter, then by the first statement, $\F|A$ is an ultrafilter for some subset $A\in\F^+$. On the other hand, the filter $\F|A$ is analytic being a continuous image of the analytic filter $\F$. So, $\F|A$ cannot be an ultrafilter.
\smallskip

4. Assume that $\F$ is an $\F_\sigma$-filter. In order to apply the preceding statement, it suffices to show that $\F$ is a $P^+$-filter. This is done in the following lemma.
\end{proof}

\begin{lemma}\label{laf} Each $F_\sigma$-filter $\F$ on $\w$ is a $P^+$-filter.
\end{lemma}

\begin{proof} According to a result of Mazur \cite{Maz} (see also \cite{Sol}), for the $F_\sigma$-filter $\F$ there exists a lower semi-continuous submeasure $\phi$ on $\mathcal P(\w)$ such that $\F=\{A\subset\w:\phi(\w\setminus A)<\infty\}$. Since $\F\ne\mathcal P(\w)$, $\phi(\w)=\infty$ and the subadditivity of $\varphi$ implies that $\phi(F)=\infty$ for all $F\in\F$. It follows from $\F=\{A\subset\w:\phi(\w\setminus A)<\infty\}$ that a set $A\subset\w$ belongs to $\F^+$ if and only if $\phi(A)=\infty$.

To show that $\F$ is a $P^+$-filter, fix any decreasing sequence of sets $(A_k)_{k\in\w}$ in $\F$.
Let $n_0=0$ and by induction construct an increasing sequence of positive integers $(n_k)_{k\in\w}$ such that
$\phi([n_k,n_{k+1})\cap A_k)>k$ for every $k\in\w$. Then the set $A=\bigcup_{k\in\w}[n_k,n_{k+1})\cap A_k$ is a pseudointersection of $(A_k)_{k\in\w}$ and belongs to the family $\F^+$ as $\phi(A)=\infty$.
\end{proof}

Let us remark that Lemma~\ref{laf} cannot be generalized to $F_{\sigma\delta}$-filters. The following example was suggested to the authors by Jonathan Verner.

\begin{example} The filter $\Fr\otimes\Fr=\{A\subset\w\times\w:\{n\in\w:\{m\in\w:(n,m)\in A\}\in\Fr\}\in\Fr\}$ on $\w\times\w$ is an $F_{\sigma\delta}$ but not $P^+$.
\end{example}

In light of Theorem~\ref{verner} it is natural to ask the following

\begin{question} Does $\beta\w$ contain an injective $\F$-convergent sequence for some analytic filter $\F$?
\end{question}

On the other hand, we have the following fact:

\begin{theorem} Each infinite compact Hausdorff space $X$ contains an injective $\F$-convergent sequence for some meager and null filter $\F$.
\end{theorem}

\begin{proof} Choose any finite-to-one function $\xi:\w\to\w$ such that $$\lim_{n\to\infty}|\xi^{-1}(n)|=\infty\mbox{ \ and \ }\prod_{n\in\w}(1-2^{-|\xi^{-1}(n)|})=0.$$ By Corollary~\ref{c3}, any infinite compact Hausdorff space $X$ contains an injective $\F$-convergent sequence for some $\xi$-meager filter $\F$. It is clear that $\F$ is meager. It remains to check that $\F$ is null. The filter $\F$, being $\xi$-meager, lies in the union $\bigcup_{n\in\w}\F_n$ where $\F_n=\{A\subset\w:\forall k\ge n\;A\cap\xi^{-1}(k)\ne\emptyset\}$. It suffices to prove that each set $\F_n$ has Haar measure zero.
Observe  that the set $\F_n$ can be identified with the product $\prod_{k\ge n}(\mathcal P(\varphi^{-1}(k))\setminus\{\emptyset\})$, which has Haar measure
$$\prod_{k\ge n} \frac{2^{|\varphi^{-1}(k)|}-1}{2^{|\varphi^{-1}(k)|}}=\prod_{k\ge n}(1-2^{-|\varphi^{-1}(k)|})=0.$$
\end{proof}

\begin{remark} After writing this paper the authors learned from V.Tkachuk that the meager property of the function space $C_p(\w^*,2)$ was also established by E.G.~Pytkeev in his Dissertation \cite[3.24]{Pyt}. Game characterizations of topological spaces $X$ with Baire function space $C_p(X,\mathbb R)$ were given in \cite{LMC} and \cite{Pyt85}.
\end{remark}

\section{Acknowledgments} The authors would like to express their thanks to Alan Dow and Jonathan Verner for very stimulating discussions and to Vladimir Tkachuk for the information about Pytkeev's results on the Baire category of function spaces.

\end{document}